\documentclass[12pt]{amsart}
\usepackage{amssymb,epic,eepic,epsfig,amsbsy,amsmath,amscd}

\usepackage{psfrag}

\numberwithin{equation}{section}
                        \theoremstyle{plain}
\usepackage{mathrsfs}

\newtheorem{theorem}{Theorem}[section]

\newtheorem{lemma}[theorem]{Lemma}
\newtheorem{corollary}[theorem]{Corollary}
\newtheorem{proposition}[theorem]{Proposition}

\theoremstyle{definition}

\def\BC{\mathbb C}

\def\BZ{\mathbb Z}

\def\SL{\mathrm{SL}}
\def\S{\mathbb S}

\def\fb{\mathfrak b}

\def\la{\langle}
\def\ra{\rangle}

\DeclareMathOperator{\tr}{\mathrm tr}

\def\ve{\varepsilon}
\def\be { \begin{equation} }
\def\ee { \end{equation} }

\begin{document}

\title[Some families of minimal elements on prime knots]
{Some families of minimal elements for the partial ordering on prime knots}

\author{Fumikazu Nagasato and Anh T. Tran}

\address{Department of Mathematics, Meijo University, 
Tempaku, Nagoya 468-8502, Japan}
\email{fukky@meijo-u.ac.jp}

\address{Department of Mathematical Sciences, The University of Texas at Dallas, 
800 W Campbell Rd, FO 35, Richardson TX 75080, USA}
\email{att140830@utdallas.edu}

\begin{abstract}
We show that all twist knots, certain double twist knots and some other 
2-bridge knots are minimal elements for the partial ordering on the set of prime knots. 
The key to these results are presentations of their character varieties 
using Chebyshev polynomials and a criterion for irreducibility of a polynomial 
of two variables. These give us an elementary method to discuss the number of 
irreducible components of the character varieties, which concludes the result essentially. 
\end{abstract}

\maketitle


\section{Outline of this research}\label{sec_intro}
We research the partial ordering on the set of prime knots by using algebraic sets 
associated to knot groups, now known as the character varieties of knot groups. 
The framework of character varieties introduced by Culler and Shalen \cite{CS} 
for a finitely presented group has been giving powerful tools 
and is now playing important roles in geometry and topology. 
On the other hand, it is not easy to calculate character varieties 
and thus to investigate their geometric structures in general, 
though an underlying idea of character varieties is simple as follows. 
Let $G$ be a finitely presented group generated by 
$n$ elements $g_1,\cdots,g_n$. For a representation $\rho:G \to \SL_2(\BC)$, 
let $\chi_{\rho}$ be the character of $\rho$, which is the function on $G$ 
defined by $\chi_{\rho}(g):=\tr(\rho(g))$ ($\forall g\in G$). 
Throughout this paper, we simply denote by $\tr(g)$ the trace $\tr(\rho(g))$ 
for an unspecified representation $\rho:G \to \SL_2(\BC)$. 
We sometimes omit the brackets in the trace like $\tr(a)=\tr a$ 
for simplicity. 
By \cite{CS} (see also \cite{GM}), {\it the $\SL_2(\BC)$-trace identity} 
\[
\tr(AB)=\tr(A)\tr(B)-\tr(AB^{-1})\hspace*{1cm} (\forall A,B\in\SL_2(\BC))
\]
shows that for any element $g \in G$, $\tr(g)$ is described by a polynomial 
in $\{\tr(g_i)\}_{1\leq i\leq n}$, $\{\tr(g_ig_j)\}_{1\leq i<j \leq n}$ and 
$\{\tr(g_ig_jg_k)\}_{1 \leq i<j<k \leq n}$. 
Then the character variety of $G$, denoted by $X(G)$, is constructed basically 
by the image of the set $\chi(G)$ of characters of $\SL_2(\BC)$-representations 
of $G$ under the map
\[
t:\chi(G) \to \BC^{n+{n\choose 2}+{n\choose 3}},\ 
t(\chi_{\rho}):=(\tr(g_i);\tr(g_ig_j);\tr(g_ig_jg_k)). 
\]
The resulting set turns out to be a closed algebraic set. 
By definition, this algebraic set depends on a choice of generators of $G$ 
(the coordinates of $X(G)$ vary if we change the choice of generating set of $G$). 
However, the geometric structures do not depend on that choice up to {\it bipolynomial map}. 
Here two algebraic sets $V$ and $W$ in some complex spaces 
are said to be isomorphic (bipolynomial equivalent) if there exist polynomial maps 
$f:V \to W$ and $g:W \to V$ such that $g \circ f=id_{V}$, $f \circ g=id_{W}$. 
We call each of $f$ and $g$ an isomorphism or a bipolynomial map. 
So $X(G)$ is an invariant of $G$ up to isomorphism (bipolynomial equivalence) of algebraic sets. 

In this research, we mainly apply {\it the Chebyshev polynomials} $S_n(z)$ 
($\forall n \in \BZ$) of the second kind defined by
\[
S_0(z)=1,\ S_1(z)=z,\ S_n(z)=zS_{n-1}(z)-S_{n-2}(z), 
\]
to describe the character varieties. 
Note that $S_{-n}(z)=-S_{n-2}(z)$ holds for any integer $n$. 
$S_n(z)$ naturally appears in the calculations of $X(G)$
since they have the similar property as the $SL_2(\BC)$-trace identity. 
For example, the third relation above of $S_n(z)$ exactly coincides with the rule 
$\tr(z^n)=\tr(z)\tr(z^{n-1})-\tr(z^{n-2})$ coming from the trace identity. 
The other Chebyshev polynomials $T_n(z)$ used in Section \ref{sec_bp3} 
also have the same property. Hence the process of calculations of 
$X(G)$ using the trace identity can be encoded into 
the Chebyshev polynomials naturally. 

Now we demonstrate a calculation of $X(G)$ using the Chebyshev polynomials $S_n(z)$ 
in the case where $G$ is {\it a knot group}. 
For a knot $K$ in $\S^3$, we denote by $G(K)$ the knot group of $K$, 
i.e., the fundamental group of the knot complement $\S^3-K$. 
For example, there exist knots, called {\it $2$-bridge knots}, each of which is parametrized by 
a sequence of integers $[a_1,a_2,\cdots,a_r]$ associated to the number of {\it twists}. 
(See Figure \ref{2-bridge_knots}. For more details, refer to \cite{Kawauchi}.)  
The knot $K_m=[-2,-m]$ depicted in Figure \ref{twist_knots} is a type of 2-bridge knots, 
called the {\it $m$-twist knot}. Sometimes the sequence $[a_1,a_2,\cdots,a_r]$ 
is encoded in a rational number $p/q$ $(p>0,-p<q<p)$ 
by the following continued fraction: 
\[
\frac{p}{q}=a_1+\cfrac{1}{a_2+\cfrac{1}{\ddots+\cfrac{1}{a_r}}}.
\]
It is shown that $p$ and $q$ may be taken to be coprime and $p$ is always odd. 
Then we denote by $\fb(p,q)$ the 2-bridge knot with a rational number expression $p/q$. 

\begin{figure}[hbtp]
\[
\psfrag{a1}{$a_1$} \psfrag{a2}{$a_2$} \psfrag{a2m-1}{$a_{2r-1}$} 
\psfrag{a}{$a$} \psfrag{b}{$b$}
[a_1,a_2,\cdots,a_{2r-1}]=
\begin{minipage}{10cm}\includegraphics[width=\hsize]{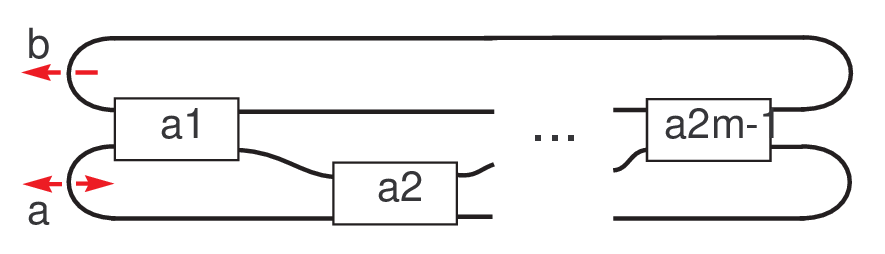}\end{minipage}
\]\[
\psfrag{a1}{$a_1$} \psfrag{a2}{$a_2$} \psfrag{a2m}{$a_{2r}$} 
\psfrag{a}{$a$} \psfrag{b}{$b$}
[a_1,a_2,\cdots,a_{2r}]=
\begin{minipage}{10cm}\includegraphics[width=\hsize]{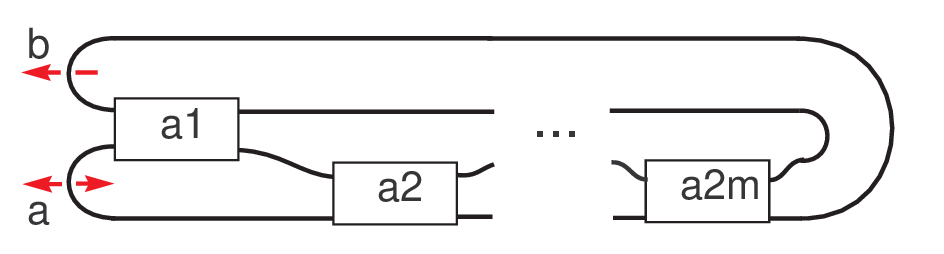}\end{minipage}
\]
\caption{The 2-bridge knot $[a_1,a_2,\cdots,a_{2r-1}]$ and $[a_1,a_2,\cdots,a_{2r}]$ 
and generators $a$ and $b$ of their knot groups. 
The orientation of $a$ is chosen so that $a$ and $b$ are conjugate. 
$a_i$ denotes the number of twists with {\it sign} in the white box.}
\label{2-bridge_knots}
\end{figure}

Note that for a 2-bridge knot $[a_1,\cdots,a_r]$ the sign of twist 
\begin{minipage}{0.6cm}\includegraphics[width=\hsize]{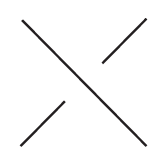}\end{minipage} 
in the white box $a_i$ is positive for $i$ odd, negative for $i$ even. 

\begin{figure}[hbtp]
\[
\psfrag{2n}{$m$} \psfrag{a}{$a$} \psfrag{b}{$b$} \psfrag{ap}{$a'$}
\psfrag{by}{$y_*$} \psfrag{y}{$y'_*$} 
K_{m}=[-2,-m]=\begin{minipage}{7cm}\includegraphics[width=\hsize]{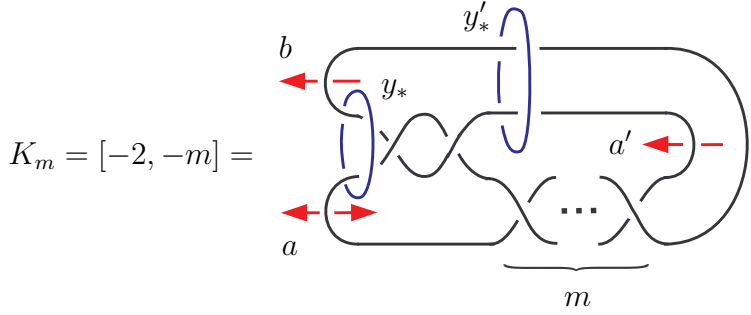}\end{minipage}
\]
\caption{The $m$-twist knot $K_{m}=[-2,-m]=\fb(2m+1,-m)$ and loops $a$, $a'$, $b$, 
$y_*$ and $y'_*$. 
The orientation of $a$ is given by the right arrow for $m$ odd, by the left arrow for $m$ even.}
\label{twist_knots}
\end{figure}

By {\it Wirtinger's algorithm} and {\it Tietze transformations} 
(reductions of generators and relations), 
we have the following presentation of the 2-bridge knot group $G(\fb(p,q))$: 
\[ G(\fb(p,q))=\la a,b \mid wa=bw \ra, \]
where $a,b$ are {\it meridians} shown in Figure \ref{2-bridge_knots}, 
$w=a^{\ve_1}b^{\ve_2} \dots a^{\ve_{p-2}}b^{\ve_{p-1}}$ and 
$\ve_j=(-1)^{\lfloor jq/p \rfloor}$. 
Here $\lfloor s \rfloor$ denotes the maximal integer $m$ satisfying $m \leq s$. 

By \cite{Le}, the character variety $X(\fb(p,q)):=X(G(\fb(p,q)))$ is given as 
the algebraic set defined by 
\[
\tr(bwa^{-1})-\tr(w)=0. 
\]
For example, in the case where $m=2n$ $(n >0)$, this equation induces the following presentation of 
$X(K_{2n}):=X(G(K_{2n}))$. 
For the twist knot $K_{2n}=\fb(4n+1,-2n)$, we have $w=u^n$ where $u=a^{-1}b^{-1}ab$. 
Let $x:=\tr(a)=\tr(b)$, $y:=\tr(ab^{-1})=\tr(y_*)$ 
and $t:=t(x,y)=\tr u=\tr(a^{-1}b^{-1}ab)=y^2-yx^2+2x^2-2$. 
We first focus on the Chebyshev polynomials $S_n(t)$. 
The following lemma is standard, see e.g. \cite[Lemma 2.2]{Tr}.
\begin{lemma}
Suppose the polynomials $f_n$ $(n \in \BZ)$ in $\BC[x,y]$ satisfy 
the recurrence relation 
$f_{n+1}=tf_n-f_{n-1}$. Then $f_n=f_0S_n(t)-f_{-1}S_{n-1}(t)$ holds.
\label{chev}
\end{lemma}
Applying Lemma \ref{chev} to the defining polynomial $\tr bwa^{-1} - \tr w$, we have 
\begin{eqnarray*}
\tr bwa^{-1} - \tr w &=& \tr bu^na^{-1} - \tr u^n\\
&=& (\tr ba^{-1}-\tr 1)S_n(t)-(\tr bu^{-1}a^{-1}-\tr u^{-1})S_{n-1}(t)\\
&=& (y-2)S_n(t)-(y-t)S_{n-1}(t)\\
&=& (y-2)\left(S_n(t)+(y+1-x^2)S_{n-1}(t)\right),
\end{eqnarray*}
since $t-y=(y-2)(y+1-x^2)$. 
Let $L_n(x,y)$ be the resulting polynomial, that is, 
\[
L_{n}(x,y)=(y-2)\left(S_n(t)+(y+1-x^2)S_{n-1}(t)\right). 
\]
Then the character variety $X(K_{2n})$ is given by 
\begin{eqnarray}\label{Ln}
X(K_{2n})=\{(x,y) \in \BC \mid L_n(x,y)=0\}. 
\end{eqnarray}

Similarly, for the twist knot $K_{2n-1}=\fb(4n-1,-2n+1)$, where $n>0$, let $u=a^{-1}b^{-1}ab$. 
Then we have $w=u^{n-1}a^{-1}b^{-1}$. 
In this case, we put $x=\tr(a)=\tr(b)$, $y=\tr(ab)=\tr(y_*)$ and 
$t:=t(x,y)=\tr(u)=\tr(a^{-1}b^{-1}ab)=y^2-yx^2+2x^2-2$. 
As in the case of $K_{2n}$, it follows from Lemma \ref{chev} that  
\begin{eqnarray*}
\tr bwa^{-1} - \tr w &=& \tr bv^{n-1}a^{-1}b^{-1}a^{-1}-\tr v^{n-1}a^{-1}b^{-1}\\
&=& (\tr ba^{-1}b^{-1}a^{-1}-\tr a^{-1}b^{-1})S_{n-1}(t)\\
&&-(\tr bv^{-1}a^{-1}b^{-1}a^{-1}-\tr v^{-1}a^{-1}b^{-1})S_{n-2}(t)\\
&=& ((\tr a^{-1})^2-\tr bab^{-1}a^{-1}-\tr ab)S_{n-1}(t)-(\tr a^{-2}-\tr b^{-1}a^{-1})S_{n-2}(t)\\
&=& (x^2-t-y)S_{n-1}(t)-(x^2-2-y)S_{n-2}(t)\\
&=& (x^2-y-2)\left((y-1)S_{n-1}(t)-S_{n-2}(t)\right).
\end{eqnarray*}
Let $L'_n(x,y)$ be the above resulting polynomial, that is,
\[
L'_n(x,y)=(x^2-y-2)\left((y-1)S_{n-1}(t)-S_{n-2}(t)\right). 
\]
Then we obtain a presentation of the character variety $X(K_{2n-1})$:
\begin{eqnarray}\label{L'n}
X(K_{2n-1})=\{(x,y) \in \BC^2 \mid L'_n(x,y)=0\}. 
\end{eqnarray}

On the other hand, we can also calculate the character varieties 
using {\it the Kauffman bracket skein module} (KBSM). 
(Refer to \cite{Bullock,P1,P2,PS}.) 
For any non-negative integer\footnote{For a negative integer $-m$ ($m>1$), 
taking the mirror image of $K_{-m}$ and arranging it, we obtain $X(K_{-m})=X(K_{m-1})$. 
In this case, $R_{-m}(x,y)$ will shift to $R_{m-1}(x,y)$.} $m$, 
let $R_m(x,y)$ be the polynomial in $\BC[x,y]$ defined by 
\[
R_m(x,y)=(y+2)\left(S_{m}(y)-S_{m-1}(y)+x^2\sum_{i=0}^{m-1}S_i(y)\right), 
\]
and let $\widetilde{R}_m(x,y)$ be the second factor of $R_m(x,y)$. 
It follows from \cite{GN} using the KBSM that $X(K_m)$ is described as 
\begin{eqnarray}\label{Rm}
X(K_m)=\left\{ (x,y)\in \BC^2 \mid R_m(x,y)=0 \right\}, 
\end{eqnarray}
where $x=-\tr(a')=-\tr(b)$ and $y=-\tr(a'b^{-1})=-\tr(y'_*)$. 
By definition, the algebraic sets in \eqref{Ln} and \eqref{Rm} at $m=2n$, 
and also the algebraic sets in \eqref{L'n} and \eqref{Rm} at $m=2n-1$ are isomorphic 
as algebraic sets. 

Expressions of $X(G)$ are sometimes quite important when we discuss geometric properties of $X(G)$. 
For example, the presentation in \eqref{Rm} is more convenient than that of \eqref{Ln} or \eqref{L'n} 
to research the number of irreducible components of $X(K_m)$ 
at least in the sense that the presentation does not depend on $m$ even or odd. 
Furthermore, we have the following from the expression in \eqref{Rm}. 

\begin{proposition}[cf. \cite{Burde, MPL}]\label{main_thm}
For any positive integer $m$, $\widetilde{R}_{m}(x,y)$ is irreducible over $\BC$. 
Therefore, $X(K_m)$ consists of two irreducible components. 
\end{proposition}

\begin{proof}
By the same argument in \cite{N2}, the factor $\widetilde{R}_m(x,y)$ 
\begin{eqnarray*}
\widetilde{R}_m(x,y)&=&S_m(y)-S_{m-1}(y)+x^2\sum_{i=0}^{m-1}S_i(y)\\
&=&S_m(y)-S_{m-1}(y)+x^2\frac{S_m(y)-S_{m-1}(y)-1}{y-2}
\end{eqnarray*}
cannot be factorized as $(h_1x+h_2)(h_3x+h_4)$ where $h_j\in \BC[y]$.  
Moreover, $(h_1x^2+h_2)h_3$ where $h_j\in \BC[y]$ cannot occur either 
as a factorization of $\widetilde{R}_m(z,y)$, since $S_m(y)-S_{m-1}(y)$ and 
$(S_m(y)-S_{m-1}(y)-1)/(y-2)$ are relatively prime in $\BC[y]$. 
Hence $\widetilde{R}_m(x,y)$ is irreducible in $\BC[x,y]$, 
concluding Proposition \ref{main_thm}. 
\end{proof}

The above method naturally leads us to Proposition \ref{irred}. 
\begin{proposition}
Suppose $\Phi(x,z)=f(z)+x^2g(z)$ is a polynomial in $\BC[x,z]$ such that 
$\deg f - \deg g$ is an odd number, and $f(z)$ and $g(z)$ are 
relatively prime in $\BC[z]$. Then $\Phi(x,z)$ is irreducible in $\BC[x,z]$. 
\label{irred}
\end{proposition}

\begin{proof}
A basic argument shows Proposition \ref{irred} likewise. 
Assume that $\Phi(x,z)$ is reducible in $\BC[x,z]$. Since $\gcd(f(z),g(z))=1$, 
we must have 
\begin{equation}
\Phi(x,z)=\left(h_1(z)+xh_2(z)\right)\left(h_3(z)+xh_4(z)\right), 
\label{h}
\end{equation}
where $h_j$'s are polynomials in $\BC[z]$. Eq. \eqref{h} is equivalent to
\[
f(z)=h_1(z)h_3(z), \quad 0=h_1(z)h_4(z)+h_2(z)h_3(z), \quad g(z)=h_2(z)h_4(z).
\]
So it follows that
\[
\deg f =\deg h_1+\deg h_3, \quad \deg h_1+\deg h_4=\deg h_2+\deg h_3, 
\quad \deg g=\deg h_2+\deg h_4.
\]
Hence 
\[
\deg f-\deg g=(\deg h_1-\deg h_2)+(\deg h_3-\deg h_4)=2(\deg h_1-\deg h_2)
\] 
is an even number, a contradiction. This proves Proposition \ref{irred}. 
\end{proof}

Proposition \ref{irred} can determine the number of irreducible components 
of the character varieties for 2-bridge knots other than twist knots. 
For example, we focus on the 2-bridge knot $\fb(p,3)$, where $p>3$ and $\gcd(p,3)=1$. 
Again, since $\fb(p,3)$ is a knot, $p$ is odd. 

\begin{figure}[hbtp]
\[
\psfrag{p=3m+1}{$p=3m+1$} \psfrag{p=3m+2}{$p=3m+2$} \psfrag{m}{$m$}
\begin{minipage}{7.5cm}\includegraphics[width=\hsize]{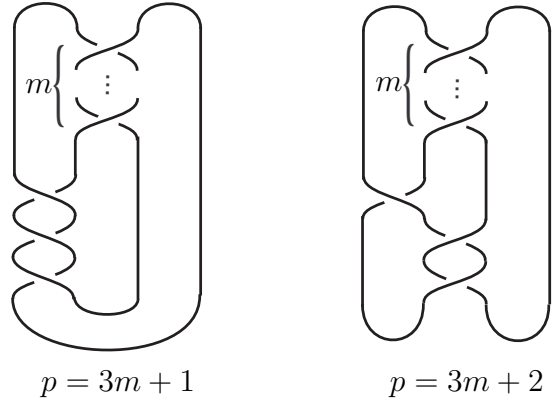}\end{minipage}
\]
\caption{The 2-bridge knot $\fb(p,3)$ with $p>3$, $\gcd(p,3)=1$. 
Since $\fb(p,3)$ is a knot, $m$ is even for $p=3m+1$ and $m$ is odd for $p=3m+2$.}
\label{bp3}
\end{figure}

As shown in Figure \ref{bp3}, 
$\fb(3m+1,3)$ is {\it a twist knot} ($m=2$) or {\it a double twist knot}, 
and $\fb(3m+2,3)$ is a twist knot ($m=1$) or a knot other than these types.
For the word $w$ in the relation of $G(\fb(p,3))$, 
let $\Phi_{w}(x,z)$ be the polynomial
\begin{eqnarray*}
\Phi_{w}(x,z) &=& S_{d}(z)-S_{d-1}(z)+ x^2(2-z) S_{d-\ell-1}(z) 
S_{\ell-1 - \lfloor \frac{\ell}{2}\rfloor}(z)
\left( S_{\lfloor \frac{\ell}{2}\rfloor}(z)
-S_{\lfloor \frac{\ell}{2}\rfloor -1}(z) \right),
\end{eqnarray*}
where $d=\frac{p-1}{2}$ and $\ell=\lfloor \frac{p}{3}\rfloor$. 
(The original definition of $\Phi_{w}(x,z)$ will appear in Section \ref{sec_bp3}.)

\begin{theorem}[Theorem \ref{thm_irr_bp3} in Section \ref{sec_bp3}, 
c.f. \cite{Burde}]\label{thm_bp3}
$X(\fb(p,3))$ with $p>3$ and $\gcd(p,3)=1$ is presented as the algebraic set defined by 
$(z+2-x^2)\Phi_{w}(x,z)=0$. 
Then $\Phi_{w}(x,z)$ is irreducible over $\BC$ and thus 
$X(\fb(p,3))$ consists of exactly two irreducible components. 
\end{theorem}

It is studied in \cite{MPL} that most double twist knots have exactly 
two irreducible components in their character varieties. 
This fact is shown by using tools in algebraic geometry. 
On the other hand, our proofs of Theorem \ref{thm_bp3} 
(i.e., Proposition \ref{m=3} and Theorem \ref{thm_irr_bp3}) 
use only basic calculations on the Chebyshev polynomials $S_n(z)$. 
This would indicate more or less an efficiency of the Chebyshev polynomials 
in the calculations of the character varieties, 
though the proofs cannot avoid laborious calculations 
(see Section \ref{sec_bp3}). 

The number of irreducible components of $X(K)$ is quite interesting 
in the sense that these results determine minimal elements for 
the partial order on the set of prime knots in $\S^3$ defined as follows. 
Let $K$ and $K'$ be prime knots in $\S^3$ which are {\it non-trivial}, 
i.e., they cannot bound embedded disks in $\S^3$. 
Then we write $K \geq K'$ if there exists an epimorphism 
(a surjective group homomorphism) from $G(K)$ onto $G(K')$. 
This defines a partial order on the set of prime knots (refer to \cite{KS}). 
We can apply the following theorem to the partial order $\geq$. 

\begin{theorem}[Theorem 4.4 in \cite{BBRW}, cf. Appendix in \cite{N2}, Corollary 7.1 in \cite{ORS}]
\label{thm-bbrw}
Suppose $K\subset \S^3$ is a hyperbolic knot in $\S^3$ 
such that $X(K)$ of $K$ has only one irreducible component 
that contains the characters of irreducible representations. 
Then $G(K)$ does not surject onto the knot group of any other non-trivial knot.
\end{theorem}

Combining Theorems \ref{main_thm} and \ref{thm-bbrw}, we obtain the following corollary. 

\begin{corollary}\label{cor_minimal}
For any positive integer\footnote{According to the property of $R_m(x,y)$ under the mirror image 
mentioned before, this naturally extends to any negative integer $m < -2$.} $m > 1$, 
at which the $m$-twist knot $K_m$ is hyperbolic, 
$K_m$ is a minimal element for the partial order $\geq$. 
\end{corollary}

Note that Corollary \ref{cor_minimal} also holds for $m=1$ (\cite{KS}), 
where $K_1$ is the trefoil knot (i.e., a non-hyperbolic knot). 
The first author has shown this corollary in the case where $2m+1$ is prime (\cite{N2}). 
Similarly, we can apply Theorems \ref{thm_bp3} and \ref{thm-bbrw} to get the following.

\begin{corollary}\label{cor_bp3}
The 2-bridge knot $\fb(p,3)$ satisfying $p>3$ and $\gcd(p,3)=1$, where it is hyperbolic, 
is a minimal element for the partial order $\geq$. 
\end{corollary}

Corollaries \ref{cor_minimal} and \ref{cor_bp3} also show the minimality 
of twist knots $K_m$ and the 2-bridge knots $\fb(p,3)$ 
with respect to the partial order 
introduced by Silver and Whitten \cite{SW} (see also \cite{HS}). 

To avoid a complicated organization for readers, we do put every laborious calculation 
in the rest of the paper. 
So, in the following section, we concentrate our focus on the calculations for 
\begin{itemize}
\item presentations of the character variety $X(\fb(p,3))$ with $p>3$ and $\gcd(p,3)=1$ 
using the Chebyshev polynomials $S_n(z)$ (Subsection \ref{subsec_31}) and, 

\item a proof of Theorem \ref{thm_bp3} using Chebyshev polynomials $S_n(z)$ 
(Subsection \ref{subsec_32}).  
\end{itemize}


\section*{Acknowledgement}
The first author had been partially supported by 
MEXT KAKENHI for Young Scientists (B) Grant Number 22740048 
and has been partially supported by JSPS KAKENHI for Young Scientists (B) 
Grant Number 26800046. 
The second author would like to thank T.T.Q. Le for helpful discussions.


\section{Character variety of $\fb(p,3)$ using $S_n(z)$: a proof of Theorem \ref{thm_bp3}}
\label{sec_bp3}
The proof of Theorem \ref{thm_bp3} consists of two parts; 
giving a description of $X(\fb(p,3))$ (Proposition \ref{m=3}) and 
the irreducibility of the polynomial $\Phi_w(x,z)$ defined below (Theorem \ref{thm_irr_bp3}), 
which describes the main body of $X(\fb(p,3))$. We state these results first. 

For the knot group $G(\fb(p,m))=\langle a,b \mid aw=wb \rangle$ of the 2-bridge knot $\fb(p,m)$, 
let $z=\tr(ab)$ and $d=(p-1)/2$. 
In general, it follows from \cite{Le} that the polynomial $\tr(bwa^{-1})-\tr(w)$, 
whose zero set coincides with the character variety $X(\fb(p,m))$, 
is described by 
\[
\tr(bwa^{-1})-\tr(w)=(z+2-x^2)\Phi_w(x,z), 
\]
where $\Phi_w(x,z)$ is the polynomial in $\BC[x,z]$ defined by 
\[
\Phi_w(x,z)=\tr w-\tr w'+\dots+(-1)^{d-1}\tr w^{(d-1)}+(-1)^d.
\] 
Here if $u$ is a word, then $u'$ denotes the word obtained from $u$ 
by deleting the two letters at the two ends. 
In general, $u^{(d-1)}$ denotes the element obtained from $u$ by applying the deleting operation 
$d-1$ times. 

In the case of $\fb(p,3)$ satisfying $p>3$ and $\gcd(p,3)=1$, 
we can describe more precisely the polynomial $\Phi_w(x,z)$ 
using the Chebyshev polynomials $S_n(z)$ as follows. 

\begin{proposition}
For the $2$-bridge knot $\fb(p,3)$ with $p>3$ and $\gcd(p,3)=1$, one has 
\begin{eqnarray*}
\Phi_{w}(x,z) &=& S_{d}(z)-S_{d-1}(z)+ x^2(2-z) S_{d-\ell-1}(z) 
S_{\ell-1 - \lfloor \frac{\ell}{2}\rfloor}(z)
\left( S_{\lfloor \frac{\ell}{2}\rfloor}(z)
-S_{\lfloor \frac{\ell}{2}\rfloor -1}(z) \right),
\end{eqnarray*}
where $\ell=\lfloor \frac{p}{3}\rfloor$.
\label{m=3}
\end{proposition}

This presentation of $\Phi_w(x,z)$ leads us to one of the main results in this paper. 

\begin{theorem}\label{thm_irr_bp3}
For the $2$-bridge knot $\fb(p,3)$ with $p>3$ and $\gcd(p,3)=1$, 
$\Phi_w(x,z)$ is irreducible in $\BC[x,z]$. 
\end{theorem}

In the rest of this section, we first show Proposition \ref{m=3}. 
Then we prove Theorem \ref{thm_irr_bp3} by Propositions \ref{irred} and \ref{m=3} 
and Lemmas \ref{roots} to \ref{gcd2} shown below. 


\subsection{Proof of Proposition \ref{m=3}}\label{subsec_31}
At first, we consider the general case $\fb(p,m)$. For $j=1, \dots, d$, let 
\[
w_j=a^{\ve_j}b^{\ve_{j+1}} \dots a^{\ve_{2d-j}}b^{\ve_{2d+1-j}}.
\] 
Then $w_1=w$ and $w_{j+1}=(w_j)'=w^{(j)}$. 
Let $u_j:=w_{j+1}a^{\ve_j}$ and $v_j:=b^{\ve_j}w_{j+1}$ 
for $j=1, \dots, d$, where $w_{d+1}:=1$. 

\begin{lemma} 
\begin{enumerate}
\item If $\ve_{j}=\ve_{j+1}$, then \begin{eqnarray*}
\tr w_j &=& z \tr w_{j+1} - \tr w_{j+2},\\
x \tr u_j &=& x^2 \tr w_{j+1}-x \tr u_{j+1},\\
x \tr v_j &=& x^2 \tr w_{j+1}-x \tr v_{j+1}.
\end{eqnarray*}

\item If $\ve_{j}=-\ve_{j+1}$, then \begin{eqnarray*}
\tr w_j=(z-x^2) \tr w_{j+1}- \tr w_{j+2}+x \tr u_{j+1}+x \tr v_{j+1}. 
\end{eqnarray*}
\end{enumerate}
\label{LT}
\end{lemma}

\begin{proof}
See \cite[Proposition A.3]{LT}.
\end{proof}

We apply Lemma \ref{LT} to describe $\tr w_j$ in $\Phi_w(x,z)$ by $S_n(z)$'s. 
For the 2-bridge knot $\fb(p,3)$, 
we can check that $\ve_j=1$ if $j \le \ell$ and $\ve_j=-1$ if $ \ell +1 \le j \le d$, 
where $\ell=\lfloor \frac{p}{3} \rfloor$.

\underline{\it Case 1}: 
$\ell +1 \le j \le d$. Since $\ve_j=\ve_{j+1}$, by Lemma \ref{LT},
\begin{eqnarray*}
\tr w_j &=& z \tr w_{j+1} - \tr w_{j+2},\\
x \tr u_j &=& x^2 \tr w_{j+1}-x \tr u_{j+1},\\
x \tr v_j &=& x^2 \tr w_{j+1}-x \tr v_{j+1}. 
\end{eqnarray*}
Note that $\tr w_d = \tr a^{\ve_d}b^{\ve_{d+1}}=\tr a^{\ve_d} b^{\ve_d}=z$ 
and $\tr w_{d+1}=\tr 1=2$. Applying the above equations recursively, we obtain 
\begin{eqnarray*}
\tr w_j &=& T_{d+1-j}(z),\\
x \tr u_j &=& x^2 \left( \tr w_{j+1} -\tr w_{j+2} + \dots 
+ (-1)^{d-1-j} \tr w_{d} \right) + (-1)^{d-j}x \tr u_d\\
&=& x^2 \left(  T_{d-j}(z)- T_{d-1-j}(z) + \dots 
+ (-1)^{d-1-j} T_1(z) + (-1)^{d-j}\right),\\
x \tr v_j &=& x^2 \left( \tr w_{j+1} -\tr w_{j+2} + \dots 
+ (-1)^{d-1-j} \tr w_{d} \right) + (-1)^{d-j}x \tr v_d\\
&=& x^2 \left(  T_{d-j}(z)- T_{d-1-j}(z) + \dots 
+ (-1)^{d-1-j} T_1(z) + (-1)^{d-j}\right),
\end{eqnarray*}
where $T_n(z)$ ($\forall n \in \BZ$) are the Chebyshev polynomials 
defined by $T_0(z)=2,\, T_1(z)=z$ and $T_{n+1}(z)=zT_n(z)-T_{n-1}(z)$.
In particular,
\begin{eqnarray*}
\tr w_{\ell+1} &=& T_{d-\ell}(z),\\
x \tr u_{\ell+1} &=& x \tr v_{\ell+1} 
= x^2 \left( T_{d-1-\ell}(z)- T_{d-2-\ell}(z) + \dots 
+ (-1)^{d-\ell-2} T_1(z) + (-1)^{d-\ell-1} \right).
\end{eqnarray*}

\underline{\it Case 2}: 
$1 \le j \le \ell-1$. Since $\ve_j=\ve_{j+1}$, by Lemma \ref{LT}, 
\[
\tr w_j= z \tr w_{j+1} - \tr w_{j+2}.
\]
It follows that $\tr w_j=S_{\ell-j}(z) \tr w_{\ell}-S_{\ell-1-j}(z) \tr w_{\ell+1}.$

\underline{\it Case 3}: $j=\ell$. Since $\ve_{\ell}=-\ve_{\ell+1}$, by Lemma \ref{LT},
\begin{eqnarray*}
\tr w_{\ell} &=& (z-x^2) \tr w_{\ell+1}- \tr w_{\ell +2}
+x \tr u_{\ell+1}+x \tr v_{\ell+1}\\
&=& (z-x^2) T_{d-\ell}(z) - T_{d-\ell-1}(z)\\
&& + \, 2 x^2 \left(  T_{d-1-\ell}(z)- T_{d-2-\ell}(z) + \dots 
+ (-1)^{d-\ell-2} T_1(z) + (-1)^{d-\ell-1} \right).
\end{eqnarray*}
Hence $\Phi_{w}(x,z)$ is equal to 
\begin{eqnarray*} 
&& \tr w_1 -\tr w_2+\dots+(-1)^{\ell -1} \tr w_{\ell} +(-1)^{\ell} \tr w_{\ell+1}+
\dots+ (-1)^{d-1} \tr w_d+(-1)^d\\
&=& \left( S_{\ell-1}(z) - S_{\ell-2}(z) + \dots + (-1)^{\ell -2} 
S_1(z)+(-1)^{\ell -1}S_0(z) \right)\tr w_{\ell}\\
&& - \, \left( S_{\ell-2}(z) - S_{\ell-3}(z) + \dots 
+ (-1)^{\ell -2} S_0(z)+(-1)^{\ell -1}S_{-1}(z) \right)\tr w_{\ell+1}\\
&& + \, (-1)^{\ell} \tr w_{\ell+1}+ \dots+ (-1)^{d-1} \tr w_d+(-1)^d\\
&=& P(z)+x^2Q(z)R(z),
\end{eqnarray*}
where
\begin{eqnarray*}
P(z) &=& T_{d-\ell+1}(z)\left( S_{\ell-1}(z) - S_{\ell-2}(z) 
+ \dots + (-1)^{\ell -2} S_1(z)+(-1)^{\ell -1}S_0(z) \right) \\
&& - \, \left( S_{\ell-2}(z) - S_{\ell-3}(z) + \dots 
+ (-1)^{\ell -2} S_0(z)+(-1)^{\ell -1}S_{-1}(z) \right)T_{d-\ell}(z)\\
&& + \, (-1)^{\ell} T_{d-\ell}(z)+ (-1)^{\ell+1} T_{d-\ell-1}(z)
+ \dots+(-1)^{d-1}T_1(z)+(-1)^d,\\
Q(z) &=& S_{\ell-1}(z) - S_{\ell-2}(z) + \dots 
+ (-1)^{\ell -2} S_1(z)+(-1)^{\ell -1}S_0(z),\\
R(z) &=& -T_{d-\ell}(z) +2 \left(  T_{d-1-\ell}(z)- T_{d-2-\ell}(z) 
+ \dots + (-1)^{d-\ell-2} T_1(z) + (-1)^{d-\ell-1} \right) .
\end{eqnarray*}

The following lemma gives us nice descriptions for $P(z)$, $Q(z)$ and $R(z)$. 

\begin{lemma} The followings hold.
\begin{enumerate}
\item $P(z) = S_d(z)-S_{d-1}(z)$,
\item $Q(z) = S_{\ell-1 - \lfloor \frac{\ell}{2}\rfloor}(z)
\left( S_{\lfloor \frac{\ell}{2}\rfloor}(z)
-S_{\lfloor \frac{\ell}{2}\rfloor -1}(z) \right)$,
\item $R(z) = (2-z)S_{d - \ell-1}(z)$.
\end{enumerate}
\label{PQR}
\end{lemma}

\begin{proof}
$(1)$ follows from \cite{Le} (see also \cite[Proposition A.2]{LT}) 
as $P(z)=\Phi_w(0,z)=S_d(z)-S_{d-1}(z)$. 
This can be checked directly by $T_j(z)=S_j(z)-S_{j-2}(z)$ and Lemma 4.3 in \cite{N1} 
saying that for any non-negative integers $r$ and $s$,
\[
S_r(u)S_{r+s}(u)=S_{2r+s}(u)+S_{2r+s-2}(u)+\cdots+S_{s}(u).
\]

To show $(2)$, let 
\[
\alpha_n=S_{n}(z) - S_{n-1}(z) + \dots + (-1)^{n-1} S_1(z)+(-1)^{n}S_0(z).
\]
Then $Q(z)=\alpha_{\ell-1}$.
If $n=2k$ then 
\begin{eqnarray*}
\alpha_n &=& \left( S_{2k}(z)+ \dots +S_0(z) \right) 
- \left( S_{2k-1}(z)+ \dots +S_1(z) \right)\\
&=& S_{k}(z)^2-S_{k}(z)S_{k-1}(z)= S_{k}(z) \left(S_{k}(z)-S_{k-1}(z) \right).
\end{eqnarray*}
If $n=2k+1$ then 
\begin{eqnarray*}
\alpha_n &=& \left( S_{2k+1}(z)+ \dots +S_1(z) \right) 
- \left( S_{2k}(z)+ \dots +S_0(z) \right)\\
&=& S_{k+1}(z)S_k(z)-S_{k}(z)^2= S_{k}(z) \left(S_{k+1}(z)-S_{k}(z) \right).
\end{eqnarray*}
In both cases $\alpha_n=S_{n - \lfloor \frac{n+1}{2}\rfloor}(z)
\left( S_{\lfloor \frac{n+1}{2}\rfloor}(z)-S_{\lfloor \frac{n-1}{2}\rfloor}(z) \right).$ 
Hence 
\[
Q(z)=\alpha_{\ell-1}=S_{\ell-1 - \lfloor \frac{\ell}{2}\rfloor}(z)
\left( S_{\lfloor \frac{\ell}{2}\rfloor}(z)
-S_{\lfloor \frac{\ell}{2}\rfloor -1}(z) \right).
\]

To show $(3)$, let
\begin{eqnarray*}
\beta_n &=& -T_{n+1}(z) 
+ 2 \left(  T_{n}(z)- T_{n-1}(z) + \dots + (-1)^{n-1} T_1(z) + (-1)^{n} \right). 
\end{eqnarray*}
Then $R(z)=\beta_{d-\ell-1}$. 
Note that $T_j(z)=S_j(z)-S_{j-2}(z)$. If $n=2k$ then  
\begin{eqnarray*}
\beta_{n} &=& -T_{2k+1}(z)+ 2+ 2 \left(\left( T_{2k}(z)+\dots+T_{2}(z)\right) 
-\left( T_{2k-1}(z)+\dots+T_{1}(z) \right)\right)\\
&=& -(S_{2k+1}(z)-S_{2k-1}(z))+2+2 \left( \left( S_{2k}(z)-S_0(z)\right) 
- \left( S_{2k-1}(z)-S_{-1}(z)\right)\right)\\
&=& -(S_{2k+1}(z)+S_{2k-1}(z))+2S_{2k}(z)\\
&=& (2-z)S_{2k}(z)=(2-z)S_n(z).
\end{eqnarray*}
If $n=2k+1$ then  
\begin{eqnarray*}
\beta_{n} &=& -T_{2k+2}(z)- 2+ 2 \left(\left( T_{2k+1}(z)+\dots+T_{1}(z)\right) 
-\left( T_{2k}(z)+\dots+T_{2}(z) \right)\right)\\
&=& -(S_{2k+2}(z)-S_{2k}(z))-2+2 \left( \left( S_{2k+1}(z)-S_{-1}(z)\right) 
- \left( S_{2k}(z)-S_{0}(z)\right)\right)\\
&=& -(S_{2k+2}(z)+S_{2k}(z))+2S_{2k+1}(z)\\
&=& (2-z)S_{2k+1}(z)=(2-z)S_n(z).
\end{eqnarray*}
In both cases $\beta_n=(2-z)S_n(z)$. 
Hence $R(z)=\beta_{d-\ell-1}=(2-z)S_{d - \ell-1}(z).$
\end{proof}

From Lemma \ref{PQR}, we get
\begin{eqnarray*}
\Phi_w(x,z) &=&P(z)+x^2Q(z)R(z)\\
&=&S_{d}(z)-S_{d-1}(z)+ x^2(2-z) S_{d-\ell-1}(z) 
S_{\ell-1 - \lfloor \frac{\ell}{2}\rfloor}(z)
\left( S_{\lfloor \frac{\ell}{2}\rfloor}(z)
-S_{\lfloor \frac{\ell}{2}\rfloor -1}(z) \right).
\end{eqnarray*}
This completes the proof of Proposition \ref{m=3}.


\subsection{Proof of Theorem \ref{thm_irr_bp3}}\label{subsec_32}
By Proposition \ref{m=3}, we have $\Phi_w(x,z)=P(z)+x^2Q(z)R(z)$, where 
\begin{eqnarray*}
P(z) &=& S_d(z)-S_{d-1}(z),\\
Q(z) &=& S_{\ell-1 - \lfloor \frac{\ell}{2}\rfloor}(z)
\left( S_{\lfloor \frac{\ell}{2}\rfloor}(z)
-S_{\lfloor \frac{\ell}{2}\rfloor -1}(z) \right),\\
R(z) &=& (2-z)S_{d - \ell-1}(z).
\end{eqnarray*}
Since $\deg P-\deg QR=d-\left( (\ell-1)+(d-\ell) \right)=1$ is an odd number, 
by Lemma \ref{irred}, $\Phi_w(x,z) \in \BC[x,z]$ is irreducible 
if $\gcd(P(z),Q(z)R(z))=1$. 

The following lemma is standard, see e.g. \cite{N2}. 

\begin{lemma}
For $n \ge 1$, the followings hold:
\begin{enumerate}
\item $S_n(z)$ is a monic polynomial of degree $n$ whose $n$ roots are exactly 
$2\cos\big(\frac{j}{n+1}\pi\big)$, $1\le j\le n$.
\item $S_n(z)-S_{n-1}(z)$ is a monic polynomial of degree $n$ whose $n$ roots 
are exactly $2\cos\big(\frac{2j+1}{2n+1}\pi\big)$, $0\le j\le n-1$.
\end{enumerate}
\label{roots}
\end{lemma}

\begin{lemma} 
$\gcd \left(S_d(z)-S_{d-1}(z), S_{\lfloor \frac{\ell}{2}\rfloor}(z)
-S_{\lfloor \frac{\ell}{2}\rfloor -1}(z) \right)=1$. 
\label{gcd1}
\end{lemma}

\begin{proof}
By Lemma \ref{roots} (2), it suffices to show that 
\begin{equation}
\frac{2j+1}{2d+1} \not= \frac{2j'+1}{2\lfloor \frac{\ell}{2}\rfloor+1}
\label{jj'}
\end{equation} 
where $0 \le j \le d-1$ and $0 \le j' \le \lfloor \frac{\ell}{2}\rfloor -1$. 
It is easy to see that Eq. \eqref{jj'} holds true 
if $\gcd \left( 2d+1, 2\lfloor \frac{\ell}{2}\rfloor+1\right)=1$.
Recall that $d=\frac{p-1}{2}$ and $\ell=\lfloor \frac{p}{3}\rfloor$. 
Since $3(2\lfloor \frac{\ell}{2}\rfloor+1)-(2d+1)$ is equal to 
either $3\ell-p$ or $3(\ell+1)-p$, and $3\ell-p=3\lfloor \frac{p}{3}\rfloor-p$ 
is equal to either $-1$ or $-2$ 
(note that $\gcd(p,3)=1$), $3(2\lfloor \frac{\ell}{2}\rfloor+1)-(2d+1)$ is equal 
to either $\pm 1$ or $\pm 2$. It follows that 
$\gcd \left( 2d+1, 2\lfloor \frac{\ell}{2}\rfloor+1\right)$ is a divisor of $2$. 
Since $2d+1$ is odd, we must have 
$\gcd \left(2d+1, 2\lfloor \frac{\ell}{2}\rfloor+1\right)=1$. 
\end{proof}

\begin{lemma}
$\gcd \left(S_d(z)-S_{d-1}(z), S_{\ell-1-\lfloor\frac{\ell}{2}\rfloor}(z)\right)
=\gcd \left(S_d(z)-S_{d-1}(z), S_{d-\ell-1}(z)\right)=1$. 
\label{gcd2}
\end{lemma}

\begin{proof}
By Lemma \ref{roots} (1), it suffices to show that 
\[
\frac{2j+1}{2d+1} \not= \frac{j'}{\ell-\lfloor \frac{\ell}{2}\rfloor},\hspace*{0.5cm}
\frac{2k+1}{2d+1} \not= \frac{k'}{d -\ell}
\]
where $0 \le j \le d-1$, $0 \le j' \le \ell - \lfloor \frac{\ell}{2}\rfloor -1$, 
$0 \le k \le d-1$ and $0 \le k' \le d -\ell - 1$. 
These hold true if $\gcd \left( 2d+1, \ell-\lfloor \frac{\ell}{2}\rfloor \right)=1$ and 
$\gcd \left( 2d+1, d -\ell \right)=1$.
Since the proof is similar to that of Lemma \ref{gcd1}, we omit the details. 
\end{proof}

We now finish the proof of Theorem \ref{thm_irr_bp3}. 
From Lemmas \ref{roots}, \ref{gcd1} and \ref{gcd2}, we have 
$\gcd \left(P(z),Q(z)R(z)\right)=1$. 
Hence Lemma \ref{irred} implies that $\Phi_w(x,z)$ is irreducible in $\BC[x,z]$ 
for the 2-bridge knot $\fb(p,3)$ and this completes the proof of 
Theorem \ref{thm_irr_bp3}. 



\begin{thebibliography}{99}

\bibitem{Bullock} D. Bullock: 
{\it Rings of $\SL_2(\BC)$-characters and the Kauffman bracket skein module}, 
Comment. Math. Helv. {\bf 72} (1997), 521--542.

\bibitem{Burde} G. Burde: 
{\it ${\rm SU}(2)$-representation spaces for two-bridge knot groups}, 
Math. Ann. \textbf{288} (1990), 103--119.

\bibitem{BBRW} M. Boileau, S. Boyer, A.W. Reid and S. Wang: 
{\it Simon's conjecture for $2$-bridge knots}, 
Comm. Anal. Geom. {\bf 18} (2010), 121--143.

\bibitem{CS} M. Culler and P. Shalen: 
{\it Varieties of group presentations and splittings of $3$-manifolds}, 
Ann. of Math. {\bf 117} (1983), 109--146.

\bibitem{GN} R. Gelca, F. Nagasato: 
{\it Some results about the Kauffman bracket skein module 
of the twist knot exterior}, 
J. Knot Theory Ramifications {\bf 15} (2006), 1095--1106. 

\bibitem{GM} F. Gonz\'{a}lez-Acu\~{n}a and J.M. Montesinos:
{\it On the character variety of group representations in $\SL(2,\BC)$ 
and $\mathrm{PSL}(2,\BC)$}, Math. Z., {\bf 214} (1993), 627--652.

\bibitem{HS} J. Hoste and P. D. Shanahan: 
{\it Epimorphisms and boundary slopes of $2$-bridge knots}, 
Algebr. Geom. Topol. {\bf 10} (2010), 1221--1244.

\bibitem{Kawauchi} A. Kawauchi: 
{\it A survey of knot theory}, Birkh\"{a}user Verlag, Basel, 1996. 

\bibitem{KS} T. Kitano and M. Suzuki: 
{\it A partial order in the knot table II}, 
Acta Math. Sin. {\bf 24} (2008), 1801--1816. 

\bibitem{Le} T.T.Q. Le: 
{\it Varieties of representations and their subvarieties of cohomology jumps 
for knot groups}, (Russian) Mat. Sb. {\bf 184} (1993), 57-82; 
translation in Russian Acad. Sci. Sb. Math. {\bf 78} (1994), 187-209.

\bibitem{LT} T. Le and A. Tran: {\it On the AJ conjecture for knots}, arXiv:1111.5258.

\bibitem{MPL}M.L. Macasieb, K.L. Petersen and R.M. van Luijk: 
{\it On character varieties of two-bridge knot groups}, 
Proc. Lond. Math. Soc. {\bf 103} (2011), 473--507.

\bibitem{N1} F. Nagasato: 
{\it Computing the A-polynomial using noncommutative methods}, J. Knot 
Theory Ramifications {\bf 14} (2005), 735--749.

\bibitem{N2} F. Nagasato: 
{\it On minimal elements for a partial order of prime knots}, 
Topology Appl. {\bf 159} (2012), 1059--1063. 

\bibitem{ORS} T. Ohtsuki, R. Riley and M. Sakuma: 
{\it Epimorphisms between $2$-bridge link groups}, 
Geom. Topol. Monogr. {\bf 14} (2008), 417--450. 

\bibitem{P1} J.H. Przytycki: 
{\it Fundamentals of Kauffman bracket skein module}, 
Kobe J. Math. {\bf 16} (1999), 45--66.

\bibitem{P2} J.H. Przytycki: 
{\it Skein modules of 3-manifolds}, 
Bull. Pol. Acad. Sci. {\bf 39} (1991), 91--100.

\bibitem{PS} J.H. Przytycki and A. Sikora: 
{\it Skein algebra of a group}, 
Banach Center Publications {\bf 42} (1998), 297--306.

\bibitem{SW} D.S. Silver and W. Whitten: 
{\it Knot group epimorphisms}, 
J. Knot Theory Ramifications {\bf 15} (2006), 153--166.

\bibitem{Tr} A. Tran: {\it The universal character ring of 
the $(-2,2m+1,2n)$-pretzel link},  Int. J. Math. {\bf 24} (2013), 
DOI: 10.1142/S0129167X13500638.
\end{thebibliography}
\end{document}